\documentclass[a4paper,10pt]{article}
\usepackage{amsmath,amssymb,graphicx}
\usepackage{url}

\renewcommand{\title}[1]{\vspace{\fill}
\eject\addtolength{\baselineskip}{4pt}
{\bfseries\LARGE #1}\\[3mm]\addtolength{\baselineskip}{-4pt}}
\renewcommand{\author}[3]{\parbox[t]{75mm}
{\begin{center}{\scshape #1}\\[3mm] #2\\
 {\ttfamily #3} \end{center}}}

\newtheorem{thm}{\bfseries Theorem}
\newtheorem{cor}[thm]{\bfseries Corollary}

\newtheorem{conj}[thm]{\bfseries Conjecture}

\newenvironment{proof}{\medskip                    
\noindent{\scshape Proof:}}{\quad $\Box$\medskip}  

\begin{document}
	
	\def\Nset{\mathbb{N}}
	\def\Ascr{\mathcal{A}}
	\def\Bscr{\mathcal{B}}
	\def\Cscr{\mathcal{C}}
	\def\Dscr{\mathcal{D}}
	\def\Escr{\mathcal{E}}
	\def\Fscr{\mathcal{F}}
	\def\Hscr{\mathcal{H}}
	\def\Iscr{\mathcal{I}}
	\def\Lscr{\mathcal{L}}
	\def\Mscr{\mathcal{M}}
	\def\Nscr{\mathcal{N}}
	\def\Pscr{\mathcal{P}}
	\def\Qscr{\mathcal{Q}}
	\def\Rscr{\mathcal{R}}
	\def\Sscr{\mathcal{S}}
	\def\Tscr{\mathcal{T}}
	\def\Uscr{\mathcal{U}}
	\def\Wscr{\mathcal{W}}
	\def\Xscr{\mathcal{X}}
	\def\cupp{\stackrel{.}{\cup}}
	\def\bold{\bf\boldmath}

\begin{center}

\title{Jump-systems of $T$-paths \footnote{To appear in the Proceedings of the Twelfth Japanese-Hungarian
		Symposium on Discrete Mathematics and Its Applications (March 2023)} }
\author{Mouna Sadli
}{
Institute for Higher Education in Morocco\\
Avenue Mohamed VI, Km 4.2 \\ (Route des Za\"ers) Souissi, Rabat, Morocco
}{
msadli@iihem.ac.ma
}
\author{
\underline{Andr\'as Seb\H{o}}
}{
Combinatorial Optimization \\
Univ. Grenoble-Alpes, CNRS, G-SCOP\\
46 Avenue F\'elix Viallet, 38240 Grenoble, France
}{
Andras.Sebo@cnrs.fr
}
%


\end{center}


\begin{quote}
{\bfseries Abstract:} Jump systems are sets of integer vectors satisfying a simple axiom, generalizing matroids, also delta-matroids, and well-kown combinatorial examples such as degree sequences of subgraphs of a graph. It is useful to know if  a set of vectors defined from combinatorial structures is a jump system:  this  has consequences for optimizing on the set, or on some derived sets of vectors.
In this note we are mainly concerned   in telling our proof of the following  more than two decades old fact and its original, elementary  proof for an example different from degree sequences: 

{\em Given an udirected graph $G=(V,E)$ and $T\subseteq V$,  the vectors $m$ indexed by  $T$ for which there exist a set of  openly disjoint $T$-paths so that each $t\in T$ is the endpoint of exactly $m(t)$ paths forms a jump system. The same holds for edge-disjoint $T$-paths.} 

We are also exhibiting the context and some consequences of this fact, with some pointers to  recent developments, among them ro another proof by Iwata and Yokoi, to some related jump system intersection theorems and to some open problems. 
\end{quote}

\begin{quote}
\end{quote}
\vspace{5mm}

\section{Introduction} 

For basic notations and terminology we refer to Schrijver \cite{SYB}. Given an undirected graph, $G=(V,E)$, and $T\subseteq V$, a $T$-path is a path $P$ such that $V(P)\cap T$ consists of the two endpoints of $P$. Two $T$-paths, $P$, $Q$ are {\em openly disjoint}, if $V(P)\cap V(Q)\subseteq T$, and they are edge-disjoint if $E(P)\cap E(Q)=\emptyset$. An integer vector $m\in\mathbb{Z}^n$ is called  {\em vertex-feasible} or {\em edge-feasible} (for $(G,T)$ )   if there exists a set of openly vertex-, resp. edge-disjoint $T$-paths so that each $t\in T$ is the endpoint of $m(t)$ of them. Edge-feasible vectors are also  called {\em node-demands}, and have been studied in  \cite{FKS}, based on which the convex hull of edge-feasible vectors can be determined. We define jump systems on $T$: 

A set $J\subseteq \mathbb{Z}^T$ is a  {\em jump system}, if for each pair $x,y\in J$ and any step $x'$ from $x$ to $y$, either $x'\in J$, or there exists a step $x''$ from $x'$ to $y$ so that $x''\in J$. A {\em step} $x'$ from $x$ to $y$ is $x':=x$ if  $x=y$, or  $x':=x+e_s$ for some $s\in T$ such that $x_s<y_s$, or   $x':=x-e_s$ for some $s\in\{1,\ldots, n\}$ such that  $x_s>y_s$.  The {\em unit vector} $e_s\in \{0,1\}^T$ is defined by $e_s(s)=1$ and $e_s(t)=0$ if $t\in T\setminus \{s\}$. 

This simple notion has been defined by Bouchet and Cunningham \cite{BC}, generalizing  the delta-matroid-axioms defined earlier by Bouchet \cite{AB}, themselves generalizing matroid axioms. Besides (delta-)matroids -- the $0-1$ special case --,  one of the best-known examples of jump-systems presented in \cite{BC} are the degree sequences of graphs. 
Another example occurred to us in 1999-2000 \cite{MA}: feasible vectors for sets of openly vertex- or edge-disjoint $T$- paths.

 Then under the impact of Schrijver's simple proof  \cite{Sch} of Mader's theorem \cite{M2} and  following his proof of the matroid property of (inclusionwise) maximal feasible vectors in a slightly different, but equivalent, $0-1$ context -- presented at the winter-school ``New Methods in Discrete Mathematics" in Alpes d'Huez, March 2000 --, we have proved that sets of vertex- and edge-feasible  vectors  form  actually   jump systems. 
 
 The publication of our proof now was encouraged by some renewed interest and deep results concerning jump systems and their intersections. First, related to $T$-paths,  by Iwata and Yokoi \cite{IYhere}, and actually a draft  of their proof of   the jump-system property in it, in February 2022, that  follows  Lov\'asz's method for proving Mader's theorem \cite{M2} which is thus very different from our proof following Schrijver proof; second, by Dudycz and Paluch's results \cite{Kasia} concerning weighted general graph factors,  generalized by \cite{K} to optimize on some particular weighted jump system intersections. These make   worth summarizing some new and old connections in Section~\ref{sec:cons}, with pointers to graph factors to analogous results for $T$-paths, and to common generalizations.   Since our proof of the jump system property of $T$-paths is not easy to access (the only public access were lectures \cite{MA} and then a French thesis \cite {MS})  we decided to make it available in this note: it is in  Section~\ref{sec:main}. Some of the not completely recent corollaries and the conjecture of Section~\ref{sec:cons} also keep the actuality of the subject with enhanced connections, until today.


\section{The jump system of feasible vectors}\label{sec:main} 


Denote by $J_{\rm vertex}(G,T)$ and   $J_{\rm edge}(G,T)$ the set of all vertex-feasible and edge-feasible vectors respectively. 

Schrijver \cite[Theorem 73.5, page 1292]{SYB} considered the matroid property of  (inclusionwise) maximal  feasible vectors. The authors were lucky enough to hear this result and Schrijver's simple proof \cite{Sch}, \cite[Theorem~73.2]{SYB} of Mader's theorems \cite{M1}, \cite{M2} -- reducing them to a result of Gallai \cite{G}, itself shortly proved from Tutte's theorem on maximum matchings -- and its corollaries, ahead of time, at the winter-school ``New Methods in Discrete Mathematics" in Alpes d'Huez, March 2000. They were strongly interested, since in 1999 they have proved weaker results  \cite{MA} in the same direction, first about the convex hull of $J_{\rm edge}(G,T)$. We will discuss some still useful  connections of these results in Section~\ref*{sec:cons}, with a related open problem. 

Mader's theorem  on the maximum number of edge-disjoint $T$-paths \cite{M1} is a straightforward consequence of the vertex-version \cite{M2} by taking the line graph,  but no easy reduction is known in the other direction. 
For proving the matroid property or the jump-system property there is no need of neither theorems though,  and there is no  essential difference between the  proof of the vertex- or edge-version. The proof of these jump-system properties is much simpler than that of Mader's theorems, and each of the vertex- or edge-versions can be obtained by mimicking the other.

Schrijver chooses the vertex-version for the proof of his matroid-property.  We choose the edge version for a difference, for the sake of introducing the possibly useful idea of ``edge-transitions'' for edge-disjoint paths,  and also because the theory of edge-disjoint paths has been much further developped than that of vertex-disjoint paths \cite{Hu}, \cite{RW}, \cite{Ch}, \cite{LCh}, \cite{Lom}, \cite{FKS}: capacities can then be put on edges, and $T$-paths can be generalized, the linear constraints for the ``node demand polyhedron" of edge-disjoint paths have been determined, so there is more to say about the edge-version.  (Similar weighted generalizations involving node-capacities are in principle possible though for the vertex-versions as well, but  the corresponding decision problems are  $\cal NP$-hard,  and  no natural analogue of the parity condition is known to ensure tractability.)

An {\em edge-transition} is an (unordered) pair of incident edges (equivalently, an edge of the line graph), or a pair $(t,e)$, where $t\in T$ and $e=tv$ $(v\in V)$, i.e.  $e$ is  an edge incident to $t$.   If $\Pscr$ is a set of paths, the union of the edge-transitions of the paths in $\Pscr$  will be denoted by $\hat\Pscr$. If $\Pscr$ consists of edge-disjoint paths, each edge-transition is contained in at most one path. The following theorem and proof  have been exposed in lectures \cite{MA}, and appear in \cite{MS}.  For the proof we introduce one more notation:
the subpath of a path $P$ between two of its points $u,v$ is denoted by $P(u,v)$.

\begin{thm}\label{thm:main} Let $G=(V,E)$ be a graph, $T\subseteq V$. Then  $J_{\rm vertex}(G,T), J_{\rm edge}(G,T)$ are   jump systems. 
\end{thm}

\begin{figure}[t]
	\centering
	\includegraphics[width=1\textwidth]{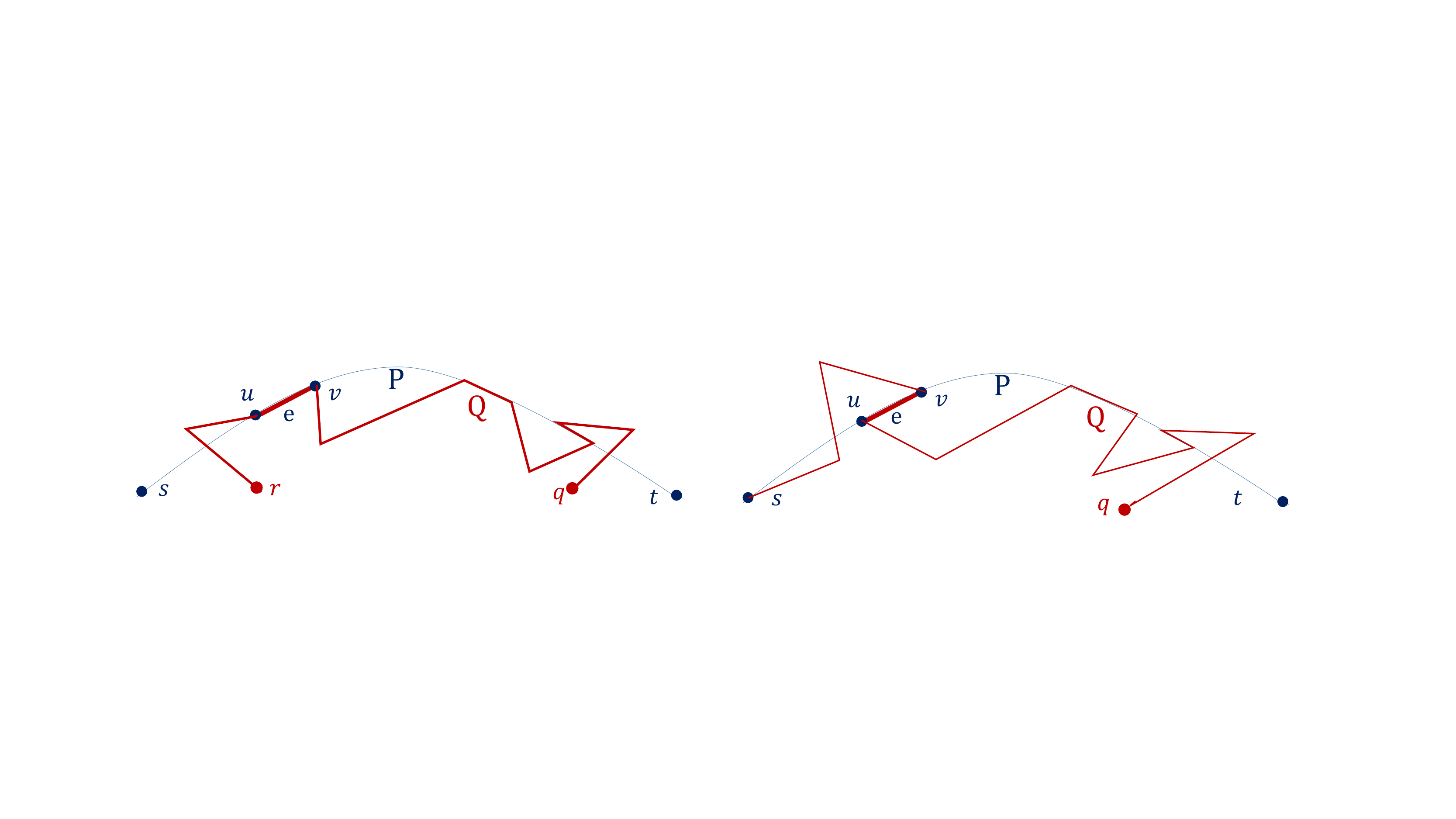}
	\vspace{-3.5cm}	\caption{ Explanation }
	\label{fig:paths}
\end{figure}

\begin{proof}  As explained above, the proof for $J_{\rm vertex}(G,T)$ and the one for  $J_{\rm edge}(G,T)$ can be obtained by mimicking one another: the most essential difference between the two is that the vertices in the former are becoming edges of the latter, and edges of the former become edge-transitions in the latter. We detail the full proof for $J_{\rm edge}(G,T)$. 
	
	Let $m_1$ and $m_2$ be two integer feasible vectors, and $\Pscr_1,\Pscr_2$ be a set of paths realizing them respectively.   We use induction with respect to $|\hat\Pscr_1\setminus \hat\Pscr_2|$ to prove that the $2$-step axiom holds for $m_1$ and $m_2$.  

\medskip\noindent
{\bf Case~1}:  The first step is $-e_s$, where $s\in T$, $m_1(s)>m_2(s)$.

We show then the $2$-step axiom for $-e_s$ as first step. If there exists an $(s,t)$-path $P\in\Pscr_1$, so that $m_1(t)>m_2(t)$, then $-e_t$ is a correct second step that is realized by deleting $P$ from $\Pscr_1$. If such a path does not exist, we delete a path having $s$ as an endpoint anyway, in the following way, keeping in mind that now $m_1(t)\le m_2(t)$ holds for the other endpoint of such a path:  

By  $m_1(s)>m_2(s)$  there exists a path  $P\in\Pscr_1$ whose first edge is incident to $s$ and  is not contained in any path of $\Pscr_2$.  Therefore, deleting $P$ from $\Pscr_1$, $|\hat\Pscr_1\setminus \hat\Pscr_2|$  decreases, enabling us to apply  the $2$-step axiom by  induction. Recalling $m_1(t)\le m_2(t)$, since after the deletion of the $P$, we get the feasible integer  vector $m_1'=m_1-e_s-e_t$, $m_1'(t)< m_2(t)$. Therefore we can apply the $2$-step axiom: for $e_t$ as first step there exists a feasible second step $\delta$.  In other words $m_1'':=m'_1+e_t+\delta=m_1-e_s-e_t+e_t+\delta=m_1-e_s+\delta$  is still a feasible integer vector,  $m_1''(t)=m_1(t)$, so $\delta$ is the second step we were looking for.

\medskip\noindent
{\bf Case~2}:  The first step is $e_s$, where $s\in T$, $m_1(s)<m_2(s)$.

We show then the $2$-step axiom for $e_s$ as first step. If there exists $t\in T$ along with an $(s,t)$-path $P\in\Pscr_2$ such that   $P$ is edge-disjoint from all paths in $\Pscr_1$, then 
\begin{itemize}
\item[-] either $m_1(t)<m_2(t)$ and then $e_t$ is a correct second step  realized by adding $P$ to  $\Pscr_1$.
\item[-] or  $m_1(t)\ge m_2(t)$, and then we show that there exists $Q\in\Pscr_1$ with endpoint $t$, and that we can apply the induction hypothesis to $(\Pscr_1\setminus \{Q\})\cup \{P\}$ to finish the proof. 
\end{itemize}
Indeed, since $P\in\Pscr_2$ is edge-disjoint from all paths in $\Pscr_1$, strictly less than  $m_1(t)$ edges incident to $t$ are used by both  $\Pscr_1$ and  $\Pscr_2$, so  there exists   $Q\in\Pscr_1$   whose edge incident to $t$ is not used by any path of $\Pscr_2$. The number of transitions of  $(\Pscr_1\setminus \{Q\})\cup \{P\}$ not contained in $\Pscr_2$ is smaller than $|\hat \Pscr_1\setminus\hat\Pscr_2|$, because $P$ is in $\Pscr_2$ so the union with $P$ does not add anything, and by the choice of $Q$, the deletion of $Q$ deletes at least one transition. 
Moreover, $\Pscr_1\setminus \{Q\}\cup \{P\}$ realizes the integer vector $m_1':=m_1+e_s-e_q$, where $q$ is the other endpoint of  $Q$, so 
\begin{itemize}
	\item[-] if  $m_1(q)\le m_2(q)$, then $m_1'(q)<m_2(q)$ and we can apply the induction hypothesis for $m_1'$, $m_2$ and $e_q$ as first step. 

With the second step $\delta$, $m_1'+e_q + \delta=m_1+e_s + \delta$ is then a feasible integer vector so $\delta$ is  a good second step for $m_1$ and $m_2$ and $e_s$ as first step, 
	\item[-] if  $m_1(q)>m_2(q)$, then the feasibility of $m_1'$ means that $-e_q$ is a good second step.
\end{itemize}
We finally suppose, still under the condition of Case~2,  that there is no path with endpoint $s$ in $\Pscr_2$, which is edge-disjoint from all paths in $\Pscr_1$. Since  $m_1(s)<m_2(s)$, there exists $P\in\Pscr_2$ with an edge incident to $s$ not used by any  path of $\Pscr_1$. By our assumption, there exists an edge $e=uv$ of $P$ also contained in a path $Q\in \Pscr_1$, which is thus not incident to $s$, and suppose that starting from $s$ on $P$, $e$ is the first such edge we meet, and that we meet $u$ before $v$. Thus neither $u$ nor $v$  are  in $T$. 

Let $q\in T$ be the endpoint of $Q$ so that $Q(u,q)$ contains $e$ (Figure~\ref{fig:paths} left), unless this endpoint is $s$, in which case define $q$ to be the other endpoint of $Q$ (Figure~\ref{fig:paths} right).  So $q\ne s$ anyway, and  $Q':=P(s,u)\cup Q(u,q)$ is a $T$-path disjoint from all paths of $\Pscr_1\setminus \{Q\}$. Now similarly to our repeated arguments, $(\Pscr_1\setminus \{Q\})\cup\{Q'\}$ shows that $m_1':=m_1 +  e_s - e_r$ is a feasible integer vector, where $r$ is the other endpoint of $Q$.  So $-e_r$ is  a correct second step, provided  $m_1(r)>m_2(r)$;  on the other hand, if $m_1(r)\le m_2(r)$,  then $m_1'(r) < m_2(r)$,  and we try to apply the $2$-step axiom with  $e_r$ as first step. For this, it is sufficient to prove the Claim below, because that allows us to apply the induction hypothesis.

We will then be done, since for $m_1', m_2,$  with a first step $e_r$ and    second step $\delta$ we have that  $m_1+e_s-e_r+e_r+\delta=m_1+e_s+\delta$ is a feasible integer vector, finishing the verification of the $2$-step axiom for Case~2, and therewith of our theorem. Indeed, then in our last case we can conclude the first step $e_s$ with the second step $\delta$. So the following Claim finishes the proof of the theorem: 

\medskip\noindent{\bf Claim}:  $|(\hat\Pscr_1\setminus \hat Q)\cup \hat Q')\setminus \hat \Pscr_2| < |\hat\Pscr_1\setminus \hat\Pscr_2|$. 

\smallskip To prove this claim note first that all edge-transitions of the path $P(s,u)$ are contained both in $Q'$ and $P\in\Pscr_2$, so they are not counted in   $|(\hat\Pscr_1\setminus \hat Q)\cup \hat Q')\setminus \hat \Pscr_2|$;  therefore there is no difference in the set of transitions of $P(s,u)$ and $Q(u,q)\subseteq Q$  between $Q$ and $Q'$. Therefore it is sufficient to examine the transitions through vertex $u$.    

Note that  the edge-transitions through $u$ decrease the induction parameter $|(\hat\Pscr_1\setminus \hat Q)\cup \hat Q')\setminus \hat \Pscr_2|$ by $1$ when   $Q$ is replaced by  $Q'$, if and only if $Q'$ uses the same edge-transition in $u$  as $P$, that is, if and only if $e\in Q'$ (Figure~\ref{fig:paths} left), and then  we are done by the induction hypothesis. 
On the other hand, by our choice,   $e\notin Q'$ happens only if $r=s$ (Figure~\ref{fig:paths} right), and then the edge-transition of $Q$ in $v$, which contains  $e$  is counted in $|\hat\Pscr_1\setminus \hat\Pscr_2|$ unless $v$ is an endpoint of $Q$,  and is no more contained   $|(\hat\Pscr_1\setminus \hat Q)\cup \hat Q')\setminus \hat \Pscr_2|$. But $v$ is indeed, not an endpoint of $Q$, since $e$ has been chosen not to be incident to $s$.    
\end{proof}

\section{Context and Consequences}\label{sec:cons}

The  consequences of Mader's theorems \cite{M1}, \cite{M2} for capacitated cases (that can be deduced by parallel edges or replications)  are well-known, and Hu's \cite{Hu}, Rothchild and Whinston's \cite{RW},  \cite{Lom}, $\ldots$, and also  later  generalizations  concern arbitrary capacities. These raise new algorithmic questions though,  beyond the size and ambitions of the present work. We therefore continue to restrict ourselves to the uncapacitated case knowing that from the viewpoint of theorems and structure the capacitated case is equivalent: for instance  an integer edge-capacity can be simulated by the same number of parallel edges. 

Let $G=(V,E)$ be a graph, and $T\subseteq V$. Definitions of feasible vectors replacing sets of openly vertex-disjoint $T$-paths by entirely vertex-disjoint paths joining  different classes of a partition $\cal T$  of $T$ are also easily seen to lead to  equivalent feasibility problems. (In terms  of combinatrial structure, while algorithmically, with binary encoding,  this needs more explanations.)  We will call such paths $\cal T$-paths. If a path is both a   $\Tscr_1$- and $\Tscr_2$-path for partitions $\Tscr_1, \Tscr_2$ of $T$, it will be  said to be a $\Tscr_1$-$\Tscr_2$-path.  The set of vertex- or edge-{\em feasible vectors for $\Tscr_1-\Tscr_2$} is defined analogously to that for $\Tscr$, leading to the study of  some particular jump system intersections. Vertex-feasible vectors for $\Tscr$ or for  $\Tscr_1-\Tscr_2$ are $0-1$ vectors and Schrijver proved that the maximal ones among them form the bases of a matroid;  
Theorem~\ref{thm:main} sharpens this to the fact that the vertex-feasible vectors for $\Tscr$ form a delta-matroid. In the rest of this article we focus on edge-feasible vectors.  

We can also define {\em relaxed-feasibility} for $\Tscr$-paths or for $\Tscr_1$-$\Tscr_2$-paths by extending the definition to not necessarily integer vectors in $\mathbb{R}^T$ with the existence of  (not necessarily integer) coefficients for each path so that the sum of coefficients containing each $e\in E$ is at most $1$.     We state here some preliminaries to Theorem~\ref*{thm:main} from \cite{MA}, \cite{MS} without proof details, but  point at  connections and  open problems related to these.   The presentation of jump-systems by these early  results is weaker than Theorem~\ref{thm:main}:   taking all  integer vectors in  the convex hull of $J_{\rm edge}(G,T)$, or the assumption of a parity condition are essential weakenings.   However, further  facts, an intersection theorem and an intriguing conjecture can be exhibited for these restricted jump systems, with interesting,  also algorithmic consequences on edge-disjoint $T$-paths. 

\begin{thm}\label{thm:FKS} Given the graph $G=(V,E)$,  $T\subseteq V$, and a partition $\cal T$ of $T$, the vertices of the polytope
	\[Q(G,\Tscr) := \{m\in \mathbb{R_+}^T:  m(X\cap T') - m(X\cap T\setminus T')\le d(X) \hbox{ for all $X\subseteq V$, $T'\in\Tscr$} \}   \]
are integer, its integer points form a jump sytem, and $Q(G,\Tscr$) is  the set of relaxed-feasible vectors. 
\end{thm}

It is easy to check that the inequalities defining $Q(G,\Tscr)$ are satisfied by any relaxed-feasible vector. 


Attention!  The   encouraging facts stated in the theorem do not imply that integer points  in $Q(G,\Tscr)$ are feasible.  Actually not all of them are, even though the membership oracle for $J_{\rm edge}(G,T)$ can be straightforwardly reduced to Mader's theorem. The situation is more difficult for the intersection of such jump systems, as we try to show  with the following results and conjecture.  

The linear inequalities describing $Q(G,\Tscr)$ are from \cite[Theorem 6.1]{FKS}, where it is also proved that the integer vectors $m\in Q(G,\Tscr)$ for which $m+d_G$ is even (in this sum $m$ is defined to be $0$ on $V\setminus T$) for all $v\in V$, are  feasible. Furthermore, the same is true for   $m\in Q(G,\Tscr_1)\cap Q(G,\Tscr_2)$, where $\Tscr_1$, $\Tscr_2$ are two partitions of $T$, (and actually even more generally). This is in the line of multiflow maximization results of  Hu \cite{Hu}, Rothschild and Whinston \cite{RW}, Cherkasski\u{\i} and Lov\'asz \cite{Ch}, \cite{LCh},  Karzanov and Lomonosov \cite{Lom} {\em under   parity constraints on the degrees}, and if the parity constraint is not supposed, only a half-integer solution can be stated. (Such a half-integer solution thus exists for all relaxed feasible vectors.) A merit of \cite{FKS} is to introduce  vectors on $T$ (``node-demands'') as an intermediate tool. Then  the goal of maximization can be achieved using matroid intersection,  implying the minmax theorems corresponding to all the mentioned results.


A {\em bisubmodular polyhedron} is a polyhedron of the form
\[Q(b):=\{x\in\mathbb{R} :\,  x(A)-x(B)\le b(A,B),\, x\ge 0\},\]
where $b$ is defined on pairs of disjoint sets  and  has values in $\mathbb{N}$, moreover it is {\em bisubmodular}, that is: 
\[f(A,B)+f(A',B')\ge f(A\cap A', B\cap B') +f((A\cup A')\setminus (B\cup B'), (B\cup B')\setminus (A\cup A')).\]

Denote by  conv$(X)$ the convex hull of the set $X\subseteq \mathbb{R}^n$.   

\begin{cor}\label{cor:bisub}
If $G$ is an arbitrary unirected graph and $T\subseteq V$, {\rm conv}$(J_{\rm vertex}(G,T))$ and {\rm conv}$(J_{\rm edge}(G,T))$ are bisubmodular polyhedra. 
\end{cor}

This shows one of the utilities of  $J$  being a jump system, having the consequence that linear objective function can be optimized with a natural  greedy algorithm in an appropriate oracle context \cite{BC} satisfied by the combinatorial examples we know about. 

\begin{proof}
By Theorem~\ref{thm:main}  $J_{\rm vertex}(G,T)$ and  $J_{\rm edge}(G,T)$ are jump systems, and Bouchet, Cunningham \cite{BC} proved that the convex hull of each jump system is a bisubmodular polyhedron. 
\end{proof}

\begin{thm}\label{thm:MA}   $Q(G,\Tscr)$ is a bisubmodular polyhedron. If all degrees of $G$ are even, then for any two partitions $\Tscr_1$, $\Tscr_2$ of $T$,   $Q(G,\Tscr_1)\cap Q(G,\Tscr_2)$ has integer vertices, and the maximum of the sum of coordinates on this intersection is achieved on a feasible vector computable in polynomial time.
\end{thm}

\begin{proof} For checking fact that $Q(G,\Tscr)$ is a bisubmodular polyhedron note first that it is defined by a $\pm 1$ constraint matrix.  Then the bisubmodular inequality can be verified directly \cite{MA}, we omit the details here, they are included in \cite{MS}.
	 
	 If all degrees are even, the bisubmodular function defining  $Q(G,\Tscr)$ has only even values. According to Cunningham \cite{C} {\em the intersection of bisubmodular polyhedra is half-integer.} The original proof of this used polyhedral arguments, and was not more difficult than  apparently the   only proof that appeared publicly, in \cite{SDONET}. It is deduced there  from a more general conjecture of Cunningham for jump systems,  and then  Cunningham's conjecture is settled using results in \cite{LLjump}.  
	
If all degrees of $G$ are even, then both bisubmodular functions  $b_1$ and $b_2$  defining $\Tscr_1$ and $\Tscr_2$ are even, so $b_i/2$ $(i=1,2)$ are integer bisubmodular functions. Applying the already established half-integrality for the intersection of the polytopes defined by $b_i/2$ $(i=1,2)$, we get that $Q(G,\Tscr_1)\cap Q(G,\Tscr_2)$ is an integer polyhedron. The last statement neecessitates a completely different proof method, it is proved in \cite[page 165 Proof of Theorem 4.3]{FKS}. 
	\end{proof}

The proof of the first statement of Theorem~\ref{thm:MA} is   easier than that of Theorem~\ref{thm:main}, which has been proved almost a year later.  The latter does not easily imply  though the former: integer vectors of  $Q(G,\Tscr)$ are not necessarily edge-feasible without the parity condition (Mader's odd cuts  also play then a role for feasibility).  Accordingly, integer vectors of  $Q(G,\Tscr)$ are not necessarily equal to conv$(J_{\rm edge}(G,T) )$.

However, if the degree of every vertex of $G$ is even, the bisubmodular functions defining $Q(G,\Tscr)$ are even,  the optimum $m$ on $Q(G,\Tscr)$ is an even vector for any objective function,  and then \cite[Theorem 6.1]{FKS} explicitly establishes the feasibility of $m$. It actually does so for all even vectors  $m\in Q(G,\Tscr_1)\cap Q(G,\Tscr_2)$. (For the $\Tscr=\Tscr_1=\Tscr_2$  special case that we need here, an easy reduction to \cite{LCh}, \cite{Ch} is actually sufficient and can be realized by adding a copy $t'$ of each terminal, and $m(t)$ paralel $tt'$ edges. The same reduction works for testing membership in the jump systems $J_{\rm edge}(G,T)$, $J_{\rm vertex}(G,T)$ in polynomial time using any algorithm for maximizing $T$-paths.)

For $\Tscr_1-\Tscr_2$-paths in graphs whose degrees are not all even,  we would need to generalize Mader's theorem to $\Tscr_1-\Tscr_2$-paths, and such a generalization does not exist since the $2$-flow problem is already $\mathcal NP$-hard. Now $m\in Q(G,\Tscr_1)\cap Q(G,\Tscr_2)$   is not even any more so \cite{FKS} cannot be applied (see \cite{FKS} for further explanations and examples). Nevertheless, for maximixing the sum of coordinates \cite{FKS} presents a patch, expressed in the     
last statement of Theorem~\ref{thm:MA}, and suggesting that  the same may also be true for {\em all vertices} of the polytope  $Q(G,\Tscr_1)\cap Q(G,\Tscr_2)$, making possible the optimization of any objective function on $\Tscr_1-\Tscr_2$-feasible vectors if the following conjecture holds:

\begin{conj}\label{conj}  If all degrees of $G$ are even, then for any two partitions $\Tscr_1$, $\Tscr_2$ of $T$, the vertices of $Q(G,\Tscr_1)\cap Q(G,\Tscr_2)$ are $\Tscr_1-\Tscr_2$-feasible, i.e.  $Q(G,\Tscr_1)\cap Q(G,\Tscr_2)$ is the convex hull of $\Tscr_1-\Tscr_2$-feasible vectors. 
\end{conj} 

If true, the assertion of this conjecture would immediately imply the polynomial computability of optimal feasible vectors for arbitrary weights on $T$. Surprisingly, this seems to be doable, without knowing whether the conjecture is true, in a general context containing both graph factors and paths (see below).  

\medskip
Another use of knowing that a set is a jump-system is that one can sometimes decide feasibility or optimize on the intersection with some other jump-systems. 
A recent breakthrough by Dudycz and Paluch's \cite{Kasia} on graph factors has been   simplified and extended by Kobayashi \cite{K} to the abstract level of jump systems.  Feasible sets defined above from disjoint paths problems can be easily proved to satisfy Kobayashi's conditions -- by pursuing Schrijver's reduction \cite{Sch} to Gallai's theorem \cite{G} where for the appropriate generalization ``general factors'' may be used instead of matchings --  to conclude with weighted optimization without confirming Conjecture~\ref{conj}.


 Kobayashi's conditions \cite[Theorem 5.1]{K} can potentially handle multiflow theorems with various kinds of weightings.  For jump systems in general,  the simple algorithm reducing the feasibility of  general factors of graphs to parity constrained factors \cite{Sthesis} (explained briefly in \cite{corn}), works for  the intersection of some jump systems \cite{S}, making possible to compute the oracle required in \cite[$C_1'$, $C_3'$]{K}, and enabling the use of \cite[Theorem 1.4]{K}, under  generalized conditions.

Results on deciding the  emptiness or finding an element of some jump-system-intersections are being explored in more details in a forthcoming article.

\smallskip\noindent {\bf Acknowledgment}: The authors are  indebted to Satoru Iwata and Yu Yokoi  for several relevant correction/update turns.  

\end{document}